\documentclass[a4paper, 12pt, oneside]{article}

\usepackage{amssymb,amsfonts,amsmath,latexsym}

\usepackage{comment}
\usepackage{epsfig}
\usepackage[font=small,labelfont=bf]{caption}
\usepackage[table]{xcolor}
\usepackage{tabularx}
\usepackage{amsthm}
\usepackage{fullpage}
\usepackage{booktabs}

\usepackage{setspace}
\usepackage{multirow}
\usepackage{multicol}

\usepackage{enumerate}
\usepackage{float}
\usepackage[stable]{footmisc}
\usepackage{titletoc}
\usepackage{caption}

\usepackage{setspace}

\usepackage[left=1in, right=1in, top=1in, bottom=1in, includefoot, headheight=13.6pt]{geometry}

\usepackage{titlesec}

\def\R{\mathbb{R}}
\def\E{\mathbb{E}}

\def\N{\mathbb{N}}

\def\n{\noindent}

\usepackage{graphicx}

\usepackage[nottoc]{tocbibind}

\usepackage{textcomp}

\newtheorem{theorem}{\sc Theorem}[section]

\newtheorem{corollary}{\sc Corollary}[section]
\newtheorem{lemma}{\sc Lemma}[section]

\newtheorem{example}{\sc Example}[section]

\title{On the capacity functional of excursion sets \\of Gaussian random fields on $\R^2$}
\author{\Large Marie Kratz* and  Werner Nagel $\dagger$\\[1ex]
\small * ESSEC Business School, CREAR \& MAP5, UMR 8145, Univ. Paris Descartes, France  \\ \small $\dagger$ Institut f\"ur Stochastik, 
Friedrich-Schiller-Universit\"at, Jena, Germany\\
\small * kratz@essec.edu;   $\dagger$ werner.nagel@uni-jena.de}

\begin{document}
\setlength{\parindent}{0in}
\parskip 2.0ex  % skip some length before paragraph
\onehalfspacing
\maketitle

% \title{{\itshape On the capacity functional of excursion sets of Gaussian random fields on $\R^2$ }
%
% \author{M. Kratz$^{\rm a}$\thanks{$^\ast$ Corresponding author. 
% Email: werner.nagel@uni-jena.de.  Marie Kratz is also member of MAP5, UMR 8145, Univ. Paris Descartes, France. 
% \vspace{6pt}} and W. Nagel$^{\rm b}$ \vspace{6pt} \\\vspace{6pt}  
% $^{\rm a}${\em ESSEC Business School Paris, CREAR, France }\\ 
% $^{\rm b}${\em Institut f\"ur Stochastik, Friedrich-Schiller-Universit\"at, D-07737 Jena, Germany}\\ \vspace{6pt}\received{ } }}

% \maketitle

\begin{abstract}
When a random field $(X_t, \ t\in {\mathbb R}^2)$ is thresholded on a given level $u$, the excursion set is given by its indicator $~1_{[u , \infty )}(X_t)$.
The purpose of this work is to study functionals (as established in stochastic geometry) of these random excursion sets, as e.g. the capacity functional as well as the second moment measure of the boundary length. It extends results obtained for the one-dimensional case %(see e.g. \cite{eik}) 
to the two-dimensional case, with tools borrowed from crossings theory, in particular Rice methods, and from integral and stochastic geometry. 

\vspace{2ex}


\vspace{2ex}

\end{abstract}

\section{ Introduction  }
\label{Sec1}

Let ${\mathbb R}^2$ denote the two-dimensional Euclidean plane with the origin $0$, the inner product $\langle\cdot ,\cdot \rangle$, the norm $||\cdot ||$ and the unit sphere ${\cal S}^1=\{ {v}\in {\mathbb R}^2: ||{v}||=1\}$. We will refer to the elements of ${\mathbb R}^2$ both as points and as vectors. The Borel $\sigma$-algebra is denoted ${\cal R}_2$.\\
Let $X$ be a stationary  random field taking values in $\R$, with $C^1$ paths. It will be described by 
$X=(X_x,\, x\in {\mathbb R}^2)$ or $(X_{s{v}}, \,s\in [0,\infty ) ,\, {v}\in {\cal S}^1)$.  We denote by $r$ its correlation function and by $f_{X_0}$ its dimension 1-marginal density function, which is a standard normal density function. \\ 

Denote by $A_u$ the excursion set of the process $X$ over a threshold $u\in\R$, i.e.
\begin{equation}\label{def_exc}
A_u = \{ x\in {\mathbb R}^2: X_x\geq u\} = \{ s{v}: X_{s{v}}\geq u, s\in [0,\infty ) ,\, {v} \in {\cal S}^1\} .
\end{equation}

Since $X$ is a random field with $C^1$ paths, then for all $u\in {\mathbb R}$, the set $A_u$ is a random closed set (see \cite{molch}, section 5.2) and the topological closure of the complement, 
denoted by $cl (A_u^c)$, is also a random closed set (see \cite{s/w}, p. 19 and Theorem 12.2.6.(b)). 
The distribution of a random closed set is fully characterized by its capacity functional $T$ (see \cite{math}, or also \cite{molch, s/w}), which for $A_u$ is defined by
\begin{equation}\label{cf}
T(K)= P(A_u \cap K \ne \emptyset), \quad \mbox{ for all } \mbox{compact subsets } K \subset \R^2.
\end{equation}
Because
$$
T(K)= P(\sup \{X_x; x\in K\} \geq u)
$$
the results for the distribution of the supremum of $X$ over a set $K$ (see e.g. \cite{ad/tay, az/wsc, pit}) can be  applied to the capacity functional. 

Often it is too complicated to describe the capacity functional completely. Therefore one usually restricts the family of sets $K$ considered in (\ref{cf}) to certain parametric families of sets, e.g. circles with varying radius or linear segments with a fixed direction and varying length.
Thus at least partial information about the distribution of the random set is available. This approach is also used in spatial statistics.

In this paper, we choose $k\geq 2$ directions given by unit vectors $v_1,\ldots,v_k\in {\cal S}^1$, and denote by $[0,l_iv_i]=\{ s\, v_i:\, 0\leq s\leq l_i  \}$ the linear segment with one endpoint in the origin 0, length $l_i>0$ and direction $v_i$. We consider the sets
\begin{equation} \label{segments}
K=\cup_{i=1}^k  [0,l_iv_i],\,\mbox{with}\, l_i\geq 0,\, i=1,\ldots ,k\,.
\end{equation} 
By $L_i=\sup \{ l:\,  [0,lv_i] \subset A_u^c\}$, we denote the random distance -- the visibility -- in direction $v_i$ from the origin 0 to the next point of the boundary $\partial A_u$, if  $0\in A_u^c$; otherwise $L_i=0$.  
The joint survival function of the visibilities can now be related to the capacity functional as
\begin{eqnarray}\label{capK}
&& P(L_1>l_1, \ldots,L_k>l_k)= P(K\subset A_u^c) = 1- T(K)  \nonumber \\
&\mbox{or}&  \\
&& T(K)= 1-P\left(X_0<u, \sup_{s\in[0, l_i]}X_{sv_i}<u, ~i=1,\ldots,k \right) \nonumber
\end{eqnarray}
The event in the last expression means that $0\in A_u^c$ and that there is no up-crossing of the process $X$ on the segments of $K$.

Besides the capacity functional of a random set, moment measures of some random measures which are induced by this set are of interest.

In the books by Adler \cite{ad:ad} , Adler and Taylor \cite{ad/tay}, Wschebor \cite{ws}, and Aza\"is and Wschebor \cite{az/wsc}, the geometry of excursion sets is studied thoroughly, in particular in  \cite{ad/tay} with explicit results for the Lipschitz-Killing curvatures (intrinsic volumes) of the excursion sets (see also \cite{ad/samo/tay}). 
In the present paper we consider the capacity functional of the excursion set for families of sets $K$ which consist of two or more linear segments, originating from a common point. This can also be interpreted as the joint distribution of the visibility in different directions from a certain point to the boundary of the excursion set. On the other hand, it can be seen as an approximation of the capacity functional of the excursion set for classes of convex polygons.

To study $T(K)$, we extend results obtained for the one-dimensional case (see e.g. \cite{eik}) to the two-dimensional case and borrow tools from  the literature on  level crossings (see \cite{ad/tay, cr/le, kr}), in particular by using Rice type methods (see \cite{az/wsc, mer, rych, ws}). We also extend an approach given in \cite{mer}, that we call the "sweeping line" method into a "growing circle" method. It will be developed in Section \ref{capa}. 

Furthermore, via our approach, we study  the second moment measure of the boundary length measure of the excursion set, provided that the boundary is smooth enough. 
If the boundary $\partial A_u$ is Hausdorff-rectifiable then with the help of the one-dimensional Hausdorff-measure ${\cal H}^1$, we define the random measure
${\cal L}$ on $[\R^2,{\cal R}_2]$ by 
$$
{\cal L} (B) = {\cal H}^1 (\partial A_u \cap B), \quad \mbox{ for all }B\in {\cal R}_2.
$$
Then the first moment measure, named also intensity measure of the random length measure, is given by 
$$
\mu ^{(1)} (B)=\E [{\cal L} (B) ]= \E [{\cal H}^1 (\partial A_u \cap B)], \, \mbox{ for all } B\in {\cal R}_2,$$ 
and the second moment measure by
$$
\mu ^{(2)} (B_1\times B_2)=\E [{\cal L} (B_1) {\cal L} (B_2)]= \E [{\cal H}^1 (\partial A_u \cap B_1) {\cal H}^1 (\partial A_u \cap B_2)],  \mbox{ for all } B_1,B_2\in {\cal R}_2.
$$
The stationarity of $X$, and thus also of $A_u$, yields that the intensity measure is a multiple of the Lebesgue measure $\lambda_2$ on $[\R^2,{\cal R}_2]$, i.e.
$\mu ^{(1)} = L_A \cdot \lambda_2$ with a positive constant $L_A$ which is the mean length of $\partial A_u$ per unit area.

Furthermore, stationarity allows the following implicit definition  of the reduced second moment measure $\kappa$ on $[\R^2,{\cal R}_2]$:
\begin{equation}\label{redsec}
\mu ^{(2)} (B_1\times B_2) = L_A^2 \int \int ~1_{B_1}(x)\, ~1_{B_2}(x+h) \, \kappa(dh)\, \lambda_2 (dx).
\end{equation}
The value $L_A\, \cdot \kappa (B)$ is the mean length of $\partial A_u$ within $B\in {\cal R}_2$, given that the origin is located at the "typical point" of the boundary (w.r.t. the length measure and the corresponding Palm distribution) (see \cite{stkm, s/w}).

Note that this second moment measure for the length of the boundary has been studied in \cite{az/wsc}  (see Theorems 6.8 and 6.9), using the co-area formula. 
Here we present an alternative approach, based on stereology,  to provide another expression for the second moment measure. 
Since this second moment measure can be determined from intersections of $\partial A_u$  with pairs of lines and from the observation of pairs of intersection points (see \cite{wei/na}), our method of counting crossings of the random field $X$ on linear segments developed in Section 2, can be applied to the estimation of the second moment measure. This will be done in Section \ref{redmom}.\\

From now on, let us assume that $X$ is Gaussian, with mean 0 and variance 1.

\section{A sweeping line and growing circle methods for an algorithmic computation of the capacity functional}\label{capa}

\noindent Sweeping line methods are well established in geometry (e.g. for the definition of the Euler-Poincar{\'e} characteristic of a set), in algorithmic geometry and in image analysis. We will apply it together with Gaussian regression  and discretization to set an algorithmic computation of the capacity functional for a pair of segments.
Then we will modify the method in order to calculate the capacity functional for a bundle of segments, using now circles with growing radius.

Suppose that $C\subset \R^2$ is a compact convex set with  $0\in C$. For $s>0$, we denote by $s\partial C= \{ sx:x\in \partial C  \}$ a homothet of the boundary of $C$, and we consider the family $(s\partial C,\, s>0)$ as a sweeping contour, determined by $C$. In this paper we will only use $C=\{ x\in \R^2: || x|| =1\}$, the boundary of the unit circle around the origin.

\subsection{The capacity functional for a bundle of two line segments} \label{2line-segments}

Consider $K$ defined in (\ref{segments}) with $k=2$, so that $K=[0, l_1v_1]\cup [0,l_2v_2]$, with $v_1\neq v_2$.

\n We also introduce  the ${\cal C}^{1}$-diffeomorphism $\rho$ (except in a finite number of points where it might only be ${\cal C}^{0}$) defined  by
\begin{eqnarray}\label{rho}
\rho:& [0,~ l_1+l_2] & \longrightarrow \quad K  \nonumber\\
& \theta &\longmapsto
\left\{
\begin{array}{ccl}
(l_1-\theta) v_1\,,& \mbox{ if } & 0\le \theta \le l_1\\
(\theta - l_1) v_2 \, ,&  \mbox{ if } & l_1\le \theta \le l_1+l_2.\\
\end{array}
\right .
\end{eqnarray}
We have, via (\ref{capK}), 
$$
P[L_1>l_1, L_2>l_2]=1-P[\sup_{s\in K}X_{s}>u] = 1-P[\sup_{\theta\in [0,l_1+l_2]}Y_{\theta}>u]
$$
where the process $Y=(Y_\theta, 0\le \theta\le l_1+l_2)$ is defined by 
$$
Y_\theta=X(\rho(\theta))\,.
$$
Let $Y'_\theta=\partial_\theta Y_\theta$ denote the derivative of $Y_\theta$ w.r.t. the parameter $\theta$.
Let  $(e_1,e_2)$  be an orthonormal basis in $\R^2$. The idea is to introduce a sweeping line parallel to the $(0e_1)$ axis, and to translate it along the  $(0e_2)$ axis until meeting a $u$-crossing by $X_s$, $s\in K$.\\  
Here we choose the $(0e_2)$ axis in such a way that the vectors $v_1$ and  $v_2$ become symmetric to the  $(0e_2)$ axis and define 
\begin{equation}\label{tildephi}
\tilde\varphi=\angle{(v_2,0e_2)} \in (0;\pi/2], \quad 
v_1=(-\sin \tilde\varphi, \cos \tilde\varphi), \quad v_2=(\sin \tilde\varphi, \cos \tilde\varphi). \quad
\end{equation}

\n We start then with the sweeping line method to express the capacity functional for a bundle of two line segments.
\begin{theorem}\label{prop2chords}
 ~  \\
Let $K=[0, l_1v_1]\cup [0,l_2v_2]$ and $\tilde\varphi$ as in (\ref{tildephi}).
The capacity functional $T$  of $A_u$ is given for $K$, as follows.\\ 
If $l_1 \le l_2$, then 
\begin{eqnarray} \label{2segmentsT1}
T(K) %=1-P[L_1>l_1, L_2>l_2] 
&=& f_{X_0}(u)\int_{[0;l_1]} \left(\E\Big[|Y'_{\theta}|~1_{\big(Y_\eta\le u, \forall \eta\in [\theta;2l_1-\theta]\big)}~/~Y_\theta=u\Big] - \right.  \nonumber\\
&& \qquad\quad \left. \E\Big[|Y'_{2l_1-\theta}|~1_{\big(Y_\eta\le u, \forall \eta\in [\theta;2l_1-\theta]\big)} ~/~Y_{2l_1-\theta}=u\Big]\right) d\theta \nonumber\\
&& + ~ f_{X_0}(u)\int_{[2l_1;l_1+l_2]} \E\Big[|Y'_{\theta}|~1_{\big(Y_\eta\le u, \forall \eta\in [0;\theta]\big)} ~/~Y_\theta=u\Big] d\theta .
\end{eqnarray}
If $l_1 \ge l_2$, then 
\begin{eqnarray} \label{2segmentsT2}
T(K)&=& f_{X_0}(u) \int_{[0;l_1-l_2]} \E\Big[|Y'_{\theta}|~1_{\big(Y_\eta\le u, \forall \eta\in [\theta; l_1+l_2]\big)} ~/~Y_\theta=u\Big] d\theta  \nonumber\\
&& + ~ f_{X_0}(u)\int_{[l_1-l_2;l_1]} \!\!\!  \left(\E\Big[|Y'_{\theta}|1_{\big(Y_\eta\le u, \forall \eta\in [\theta; 2l_1-\theta]\big)} /Y_\theta=u\Big] ~ - \right.  \nonumber\\
&& \qquad \qquad \quad \left. \E\Big[ |Y'_{2l_1-\theta}|1_{\big(Y_\eta\le u, \forall \eta\in [\theta; 2l_1-\theta]\big)} /Y_{2l_1-\theta}=u\Big]   \right) d\theta . \quad
\end{eqnarray}
~
\end{theorem}

\begin{proof}
As already mentioned, we introduce a sweeping line parallel to the $(0e_1)$ axis and translate it along the  $(0e_2)$ axis until meeting a $u$-crossing by $X_s$, $s\in K$.
Setting 
$$
\Gamma_t =\{s=(s_1,s_2)\in K: s_2\le t_2\}, \; t=(t_1,t_2)\in \R^2,
$$ 
where the parameter $t_2$ indicates the position of that sweeping line, we can write
$$
P[L_1>l_1, L_2>l_2]=1-\E[\# \{\theta\in[0,l_1+l_2], Y_\theta=u, X_s\le u, \forall s\in\Gamma_{\rho(\theta)}\}]  
$$
where $\displaystyle \# \{\theta\in[0,l_1+l_2], Y_\theta=u, X_s\le u, \forall s\in\Gamma_{\rho(\theta)}\}= 1$ if there is a (first) crossing by $X$ on $K$, and 0 otherwise.\\ 
So, using Rice formula ($f_{Y_\theta}$ denoting the density function of $Y_\theta$), then the stationarity of $X$, we obtain
\begin{eqnarray} \label{expectY}
P[L_1>l_1, L_2>l_2]&=& 1-\int_0^{l_1+l_2} \E\big[ |Y'_{\theta}|~1_{(X_s\le u, \forall s\in\Gamma_{\rho(\theta)})} ~/~Y_\theta=u\big] f_{Y_\theta}(u)d\theta\nonumber\\
&=& 1-f_{X_0}(u)\int_0^{l_1+l_2} \E\big[ |Y'_{\theta}|~1_{(X_s\le u, \forall s\in\Gamma_{\rho(\theta)})} ~/~Y_\theta=u\big] d\theta\, . \qquad 
\end{eqnarray}
Note that this type of integrals can be numerically evaluated as in  \cite{mer}.\\
Let us go further in the study of the integral appearing in (\ref{expectY}), reducing the problem to a one-dimensional parameter set. \\
If $l_1 \le l_2$, then 
\begin{eqnarray}  \label{integrandl1<l2}
&&\int_{[0;l_1+l_2]} \E\big[ |Y'_{\theta}|~1_{\big(X_s\le u, \forall s\in\Gamma_{\rho(\theta)}\big)} ~/~Y_\theta=u\big] d\theta ~= \\
&& \int_{[0;l_1]} \left(\E\Big[|Y'_{\theta}|~1_{\big(Y_\eta\le u, \forall \eta\in [\theta;2l_1-\theta]\big)} ~/~Y_\theta=u\Big]
-  \E\Big[|Y'_{2l_1-\theta}|~1_{\big(Y_\eta\le u, \forall \eta\in [\theta;2l_1-\theta]\big)} ~/~Y_{2l_1-\theta}=u\Big]\right)d\theta  \nonumber\\
&& +~\int_{[2l_1;l_1+l_2]} \E\Big[|Y'_{\theta}|~1_{\big(Y_\eta\le u, \forall \eta\in [0;\theta]\big)} ~/~Y_\theta=u\Big] d\theta . \nonumber
\end{eqnarray}
If $l_1\ge l_2$, then 
\begin{eqnarray} \label{integrandl1>l2}
&&\int_{[0;l_1+l_2]} \E\big[ |Y'_{\theta}|~1_{\big(X_s\le u, \forall s\in\Gamma_{\rho(\theta)}\big)} ~/~Y_\theta=u\big] d\theta ~=\\
&& \int_{[0;l_1-l_2]} \E\Big[|Y'_{\theta}|~1_{\big(Y_\eta\le u, \forall \eta\in [\theta; l_1+l_2]\big)} ~/~Y_\theta=u\Big] d\theta ~+ \nonumber\\
&& \int_{[l_1-l_2;l_1]} \!\!\!  \left(\E\Big[|Y'_{\theta}|1_{\big(Y_\eta\le u, \forall \eta\in [\theta; 2l_1-\theta]\big)} /Y_\theta=u\Big] 
- \E\Big[ |Y'_{2l_1-\theta}|1_{\big(Y_\eta\le u, \forall \eta\in [\theta; 2l_1-\theta]\big)} /Y_{2l_1-\theta}=u\Big]   \right) d\theta\,. \nonumber \\
~ \nonumber
\end{eqnarray}
Hence the result.
\end{proof}
%%%%%%%%%%%%
%%%%%%%%%%%%

\n Let $I(\theta)$ denote the following interval (as it appears in the indicator functions of  (\ref{2segmentsT1}) and (\ref{2segmentsT2})):
\begin{equation}\label{Itheta}
I(\theta) = \left\{
\begin{array}{lcccl}
~ [ \theta, 2l_1-\theta] , & \mbox{ for } & l_1\le l_2 &,&  0\le \theta \le l_1\, , \\
~ [ 0,\theta], & \mbox{ for } & l_1\le l_2 &,&   2l_1\le \theta\le l_1+l_2\, , \\
~ [0, l_1+l_2], & \mbox{ for } & l_1> l_2 &,&   0 \le \theta\le l_1-l_2 \, , \\ 
~ [\theta, 2l_1-\theta ], & \mbox{ for } & l_1> l_2 &,&  l_1-l_2 <\theta\le l_1\, .
\end{array}
\right.
\end{equation}
 
\n The integrands appearing  in Theorem~\ref{prop2chords} as conditional expectations of the form \\
$\displaystyle \E\Big[|Y'_{\theta}|~1_{\big(Y_\eta\le u,\, \forall \eta\in I(\theta)\big)} ~/~Y_\theta=u\Big]$
will now be treated via an approximation by discretization. We will use a standard method when working with Gaussian vectors, namely the Gaussian regression (see e.g. \cite{lind}). 
This may allow to handle numerically the computation of the conditional expectations. \\

Before stating the main result, let us introduce some further notation.\\
Let $\partial v_i$  denote the directional derivative w.r.t. $v_i$, for $i=1,2$, which corresponds to
\begin{eqnarray*} 
\partial v_1 X_{lv_1}&=& \lim_{h\to 0}\frac1h \big(X_{(l+h)v_1}-X_{l v_1}\big)\\
&=& - \sin\tilde\varphi ~\partial_{10}X_{-l\sin\tilde\varphi,~l\cos\tilde\varphi} +
\cos\tilde\varphi ~\partial_{01} X_{-l\sin\tilde\varphi,~l\cos\tilde\varphi}
\end{eqnarray*} 
$$
\mbox{and}\quad \partial v_2 X_{lv_2}=  \sin\tilde\varphi ~\partial_{10}X_{l\sin\tilde\varphi,~l\cos\tilde\varphi} +
\cos\tilde\varphi ~\partial_{01} X_{l\sin\tilde\varphi,~l\cos\tilde\varphi}
$$
where $\partial_{ij}$ denotes the partial derivative of order $i+j$ with $i$th partial derivative in direction $e_1$ and  $j$th partial derivative in direction $e_2$.\\
Recall that the covariances between the process $X$ and its partial derivatives, when existing, are given, for $s,t,h_1,h_2\in \R^2$, by (see \cite{kl})
\begin{eqnarray}\label{partial-r}
E\left[\partial_{jk}X_{s+h_1,t+h_2} \cdot \partial_{lm}X_{s,t}\right]
&=& (-1)^{l+m}\partial_{j+l,k+m}r(h_1,h_2), \\
&& \mbox{for all  }  0\le j+k\le 2, ~ 0\le l+m\le 2 . \nonumber
\end{eqnarray}

\begin{theorem}\label{RegresDiscrete}
Let $X$ be a stationary Gaussian random field, mean 0 and variance 1, with $C^1$ paths
and a  twice differentiable correlation function $r$.  
Further, for all $m\in {\mathbb N}$, let $\eta_1,\ldots ,\eta_m$ be equidistant points, partitioning $I(\theta )$ (defined in (\ref{Itheta})), into $m-1$ intervals (where $\eta_1$ and $\eta_m$ coincide with the left and right boundary of $I(\theta)$, respectively).
Then we have
\begin{eqnarray}\label{regresDiscreteEq}
&&  \E\Big[|Y'_{\theta}|~1_{\big(Y_\eta\le u, \forall \eta\in I(\theta)\big)} ~/~Y_\theta=u\Big] \\
&& =  \lim_{m\to\infty}
 \int_\R |y|  F_{\xi^{(m)}}\left(u\big(1-a(\eta_i;\theta)\big)-y ~b(\eta_i;\theta); ~i=1,\ldots,m  \right)f_{Y'_{\theta}}(y)dy \nonumber
\end{eqnarray}
where the density $f_{Y'_{\theta}}$ of $Y'_{\theta}$ is Gaussian  with mean $0$ and variance given by
 \begin{eqnarray*}
\E(Y'^{~2}_\theta)&=& \left( -\partial_{20}r(0,0)\sin^2\tilde\varphi  - \partial_{02}r(0,0)  \cos^2\tilde\varphi + 2  \partial_{11}r(0,0)  \sin\tilde\varphi \cos\tilde\varphi\right) 1_{(0\le \theta < l_1)}\\
&& -  \left(\partial_{20}r(0,0)\sin^2\tilde\varphi  + \partial_{02}r(0,0)  \cos^2\tilde\varphi + 2  \partial_{11}r(0,0)  \sin\tilde\varphi \cos\tilde\varphi\right)1_{(l_1< \theta \le l_1+l_2)},
 \end{eqnarray*}
 $\tilde\varphi$ being defined in (\ref{tildephi}),
and  where $F_{\xi^{(m)}}$ is the cdf of the Gaussian vector $\xi^{(m)}= (\xi_i, ~i=1,\cdots,m)$:  ${\cal N}(0, \Sigma_m)$ with the covariance matrix $\Sigma_m$ given  by
$$
var(\xi_i)= 1-a^2(\eta_i,\theta)-b^2(\eta_i,\theta) 
$$
and, for  $\eta_i, i=1,\ldots,m$ pairwise different,
$$
cov(\xi_i,\xi_j)= a(\eta_i,\eta_j)-a(\eta_i,\theta)a(\eta_j,\theta)-b(\eta_i,\theta)b(\eta_j,\theta)\E(Y'^{~2}_\theta)
$$
the coefficients $a(.,.)$ and $b(.,.)$ being defined below in  (\ref{a}) and (\ref{b}) respectively. \\~
\end{theorem}

We can deduce from this theorem an approximation quite useful for a numerical evaluation of the capacity functional, namely:
\begin{corollary}\label{numericalApprox}
The capacity functional $T(K)$ given in Theorem~\ref{prop2chords} can be numerically evaluated by approximating, for large $m$, its integrands as:
\begin{eqnarray}\label{approximEq}
&&\E\Big[|Y'_{\theta}|~1_{\big(Y_\eta\le u, \forall \eta\in I(\theta)\big)} ~/~Y_\theta=u\Big] \\ 
&& \approx  \int_\R |y|  F_{\xi^{(m)}}\left(u\big(1-a(\eta_i;\theta)\big)-y ~b(\eta_i;\theta); ~i=1,\ldots,m  \right)f_{Y'_{\theta}}(y)dy \,. \nonumber\\
&&\nonumber
\end{eqnarray}
\end{corollary}

The proof of Theorem~\ref{RegresDiscrete}  is based on the following lemma.
\begin{lemma}\label{khintchinExt}
Under the assumptions of Theorem~\ref{RegresDiscrete}, we have
$$
 \E\Big[|Y'_{\theta}|~1_{\big(Y_\eta\le u, \forall \eta\in I(\theta)\big)} ~/~Y_\theta=u\Big] = \lim_{m\to\infty}  \E\Big[|Y'_{\theta}|~1_{\big(Y_{\eta_1}\le u, \ldots, Y_{\eta_m}\le u \big)} ~/~Y_\theta=u\Big]\, .
$$
\end{lemma}

\begin{proof} 
Let $\displaystyle D_i^{(m)} = \{C_u(I_i^{(m)})\ge 2 \}$ denote the event that the number of crossings in the interval $I_i^{(m)}$, $i=1,\ldots ,m-1$, is larger or equal than 2, where $I_i^{(m)}$ is the $i$th open interval of the equidistant partition of $I(\theta)$ into $m-1$ intervals by $\eta_1,\ldots,\eta_m$.

Noticing that 
$$
1_{( Y_{\eta_1}\le u, \ldots, Y_{\eta_m}\le u )} - 1_{\bigcup_{i=1}^{m-1}D_i^{(m)}} \leq 1_{(Y_\eta\le u,\, \forall \eta\in I(\theta))}  \leq 1_{( Y_{\eta_1}\le u, \ldots, Y_{\eta_m}\le u )},
$$
we can write
\begin{eqnarray}\label{indicat}
 \E\Big[|Y'_{\theta}|~\Big(1_{( Y_{\eta_1}\le u, \ldots, Y_{\eta_m}\le u )}- 1_{\bigcup_{i=1}^{m-1}D_i^{(m)}} \Big) ~/~Y_\theta=u\Big] 
 &\le &   \E\Big[|Y'_{\theta}|~1_{(Y_\eta\le u, \forall \eta\in I(\theta))} ~/~Y_\theta=u\Big] \nonumber \\
\! \!\!&\le & \E\Big[|Y'_{\theta}|~1_{( Y_{\eta_1}\le u, \ldots, Y_{\eta_m}\le u )} ~/~Y_\theta=u\Big] . \qquad
 \end{eqnarray}
Moreover, since 
$\displaystyle \, \forall m\in \N , \; |Y'_\theta| \,1_{\bigcup_{i=1}^{m-1} D_i^{(m)}} \, \le \, |Y'_\theta | 
$,
and  $|Y'_\theta|$ is integrable w.r.t. the conditional distribution given $Y_\theta=u$ (the number of crossings in $I(\theta)$ having finite mean),
then,  using the theorem of dominated convergence, we obtain
$$
 \lim_{m\to\infty}   \E[|Y'_\theta |  1_{\bigcup_{i=1}^{m-1} D_i^{(m)}} \,/\, Y_\theta = u] = \E[|Y'_\theta | \lim_{m\to\infty}  1_{\bigcup_{i=1}^{m-1} D_i^{(m)}} \,/\, Y_\theta = u]\,.
$$
Since  $\displaystyle \lim_{m\to\infty}  1_{\bigcup_{i=1}^{m-1} D_i^{(m)}}=0$ for almost all paths of $Y$, we can deduce that
\begin{equation}\label{domina}
 \lim_{m\to\infty}   \E[|Y'_\theta |  1_{\bigcup_{i=1}^{m-1} D_i^{(m)}} \,/\, Y_\theta = u] =0\,.
\end{equation}
Combining (\ref{indicat}) and (\ref{domina}) allows to conclude that
$$
\E\Big[|Y'_{\theta}|~1_{(Y_\eta\le u, \,\forall \eta\in I(\theta))} ~/~Y_\theta=u\Big] =\lim_{m\to \infty}  \E\Big[|Y'_{\theta}|~1_{( Y_{\eta_1}\le u, \ldots, Y_{\eta_m}\le u )} ~/~Y_\theta=u\Big]\, . 
$$
\end{proof}

{\bf Proof of Theorem~\ref{RegresDiscrete}:}
Regressing the random vector $Y^{(m)}_\eta=(Y_{\eta_1}, \cdots,Y_{\eta_m})$, $m\ge 1$, on $Y_\theta$ and $Y'_\theta$ , which  are independent at fixed $\theta$ (see e.g. \cite{cr/le}), gives
\begin{equation}\label{regression}
Y^{(m)}_\eta=\delta^{(m)} ~\xi^{(m)} +a^{(m)} Y_\theta+b^{(m)}Y'_\theta 
\end{equation}
where the deterministic vectors $\delta^{(m)} = (\delta(\eta_1,\theta),\cdots, \delta(\eta_m,\theta))$, \\
$a^{(m)} = (a(\eta_1,\theta),\cdots, a(\eta_m,\theta))$ and $b^{(m)} = (b(\eta_1,\theta),\cdots, b(\eta_m,\theta))$ have their components defined respectively by 
$$
\delta(\alpha,\theta)= 1_{(\alpha\neq \theta)}; \qquad a(\theta,\theta)=1; \qquad b(\theta,\theta)=0
$$
and, for $\alpha\neq \theta$,
\begin{eqnarray} \label{a}  
a(\alpha,\theta)&=& \E[Y_\alpha Y_\theta] \\
&=& \left\{
\begin{array}{lcl}
r\big( (\theta-\alpha)\sin\tilde\varphi, (\alpha-\theta)\cos\tilde\varphi \big) &\mbox{if} & 0\le \theta,\alpha\le l_1\\
r\big((\theta-\alpha) \sin\tilde\varphi , (\theta-\alpha)\cos\tilde\varphi \big) &\mbox{if} & \theta,\alpha \ge l_1\\
r\big(  (2l_1-\alpha-\theta)\sin\tilde\varphi, (\theta-\alpha)\cos\tilde\varphi\big) &\mbox{if} & 0\le \theta\le l_1\le \alpha\le l_1+l_2\\
r\big( (2l_1-\alpha-\theta) \sin\tilde\varphi, (\alpha-\theta)\cos\tilde\varphi \big) &\mbox{if} & 0\le \alpha\le l_1\le \theta\le l_1+l_2
\end{array} \nonumber 
\right.
 \end{eqnarray}
 \begin{eqnarray}\label{b}
&& b(\alpha,\theta)~ = ~ \E[Y_\alpha Y'_\theta] = \E[Y_\alpha \partial_{v_1}Y_\theta] 1_{(\theta\in[0,l_1])}+ \E[Y_\alpha \partial_{v_2}Y_\theta] 1_{(\theta\in(l_1,l_1+l_2])} \\
\!\!\!\!\! &\!\!\!\!\! =\!\!\!\!\!& \left\{ 
\begin{array}{l}
- \sin\tilde\varphi\partial_{10}r\big((\theta-\alpha)\sin\tilde\varphi,(\alpha-\theta)\cos\tilde\varphi\big)
+ \cos\tilde\varphi\partial_{01}r\big((\theta-\alpha)\sin\tilde\varphi,(\alpha-\theta)\cos\tilde\varphi\big) ~ \mbox{if} ~ 0\le \theta,\alpha\le l_1\\
%%%
\sin\tilde\varphi ~\partial_{10}r\big( (\alpha-\theta)\sin\tilde\varphi, (\alpha-\theta) \cos\tilde\varphi \big) 
+\cos\tilde\varphi ~\partial_{01}r\big( (\alpha-\theta)\sin\tilde\varphi, (\alpha-\theta)\cos\tilde\varphi\big)  ~ \mbox{if} ~ \theta,\alpha \ge l_1\\
%%%
\sin\tilde\varphi\partial_{10}r\big((\alpha-\theta)\sin\tilde\varphi,(\alpha+\theta-2l_1)\cos\tilde\varphi\big)
- \cos\tilde\varphi\partial_{01}r\big((\alpha-\theta)\sin\tilde\varphi,(\alpha+\theta-2l_1)\cos\tilde\varphi\big)  \\
%%%
\qquad \qquad\qquad \qquad\qquad \qquad\qquad\qquad \qquad \qquad \qquad \qquad\qquad\quad \mbox{if} \quad 0\le \theta\le l_1\le \alpha\le l_1+l_2\\
%%%
\sin\tilde\varphi ~\partial_{10}r\big( (\theta-\alpha)\sin\tilde\varphi, (\alpha+\theta-2l_1)\cos\tilde\varphi \big) 
+\cos\tilde\varphi ~\partial_{01}r\big( (\theta-\alpha)\sin\tilde\varphi, (\alpha+\theta-2l_1) \cos\tilde\varphi \big) \\
%%%
\qquad \qquad\qquad \qquad\qquad \qquad\qquad\qquad \qquad \qquad \qquad \qquad\qquad\quad \mbox{if} \quad 0\le \alpha\le l_1\le \theta\le l_1+l_2
\end{array}\right.  \nonumber
 \end{eqnarray}
and where the random vector $\xi^{(m)}= (\xi_1,\cdots,\xi_m)$ is independent of  $(Y_\theta, Y'_\theta)$, Gaussian ($F_{\xi^{(m)}}$ denoting its cdf), mean $0$, covariance matrix $\Sigma_m$ with 
$$
var(\xi_i)= 1-a^2(\eta_i,\theta)-b^2(\eta_i,\theta) \quad (i=1,\cdots,m),
$$
and, for $\eta_1,\cdots,\eta_m$ pairwise different,
$$
cov(\xi_i,\xi_j)=
\E(\xi_i\xi_j)= a(\eta_i,\eta_j)-a(\eta_i,\theta)a(\eta_j,\theta)-b(\eta_i,\theta)b(\eta_j,\theta)\E(Y'^{~ 2}_\theta)
$$
since  $\displaystyle \E(Y_\theta^2)= var(X_{\rho(\theta)})=1$.
Using (\ref{partial-r})  gives,\\
on one hand, if $0\le \theta < l_1$,
 \begin{eqnarray*}
\E(Y'^{~2}_\theta)&=& \E[(\partial_{v_1}X_{(l_1-\theta)v_1})^2]\\
&=&\E\Big[\left(-\sin\tilde\varphi ~ \partial_{10}X_{-(l_1-\theta)\sin\tilde\varphi,~(l_1-\theta)\cos\tilde\varphi} + 
\cos\tilde\varphi ~\partial_{01}X_{-(l_1-\theta)\sin\tilde\varphi,~(l_1-\theta)\cos\tilde\varphi} \right)^2\Big]\\
&= & -\partial_{20}r(0,0)\sin^2\tilde\varphi  - \partial_{02}r(0,0)  \cos^2\tilde\varphi + 2  \partial_{11}r(0,0)  \sin\tilde\varphi \cos\tilde\varphi
 \end{eqnarray*}
and, on the other hand, if $l_1< \theta \le l_1+l_2$,
  \begin{eqnarray*}
\E(Y'^{~2}_\theta)&= & -\partial_{20}r(0,0)\sin^2\tilde\varphi  - \partial_{02}r(0,0)  \cos^2\tilde\varphi - 2  \partial_{11}r(0,0)  \sin\tilde\varphi \cos\tilde\varphi\,.
 \end{eqnarray*}
 Therefore, using this Gaussian regression for any vector $Y^{(m)}_\eta$ of any size $m$, and the independence of ($Y_\theta, Y'_\theta, \xi$),
we can write, for the interval $I(\theta)$, $\xi=(\xi_\eta)$ denoting the Gaussian process defined by its finite dimensional distributions (fidis) of $\xi^{(m)}$,
\begin{eqnarray*}
&& \E\Big[|Y'_{\theta}|~1_{\big(Y_\eta\le u, \, \forall \eta\in I(\theta)\big)} ~/~Y_\theta=u\Big] \\
&  =& \E\big[ |Y'_{\theta}|~1_{\big(b(\eta,\theta)Y'_\theta\le u(1-a(\eta,\theta))-\delta(\eta,\theta)\xi_\eta, ~\forall \eta\in I(\theta)\big)} \big] \\
 & =& \E\Big(\E\big[ |Y'_{\theta}|~1_{\big(b(\eta,\theta)Y'_\theta\le u(1-a(\eta,\theta))-\delta(\eta,\theta)\xi_\eta, ~\forall \eta\in I(\theta)\big)} \big]~/~\xi\Big)\,.
  \end{eqnarray*}
To compute this last expression, we proceed by discretization, working on vectors. We have, for a given vector $(\eta_1,\ldots,\eta_m)$, \\
 \begin{eqnarray}\label{UBintegrand}
 && \E\Big[|Y'_{\theta}|~1_{\big(Y_{\eta_1}\le u, \ldots, Y_{\eta_m}\le u \big)} ~/~Y_\theta=u\Big] \nonumber\\
 &=&  \int _{\R^m}\E\big[ |Y'_{\theta}|~1_{\big(b(\eta_i;\theta)Y'_\theta\le u(1-a(\eta_i;\theta))-z_i; ~i=1,\ldots,m\big)}
  ~/ ~\xi^{(m)}=z \big] f_{\xi^{(m)}}(z)dz  \nonumber\\
 &=& \int _{\R^m}\E\big[ |Y'_{\theta}|~1_{\big(b(\eta_i;\theta)Y'_\theta\le u(1-a(\eta_i;\theta))-z_i; ~i=1,\ldots,m\big)} \big] f_{\xi^{(m)}}(z)dz  \nonumber\\
 &=& \int_\R |y|  \int _{\R^m} 1_{\big(z_i \le (1-a(\eta_i;\theta))u-b(\eta_i;\theta)y; ~i=1,\ldots,m\big)} f_{\xi^{(m)}}(z)dz f_{Y'_{\theta}}(y)dy  \nonumber\\
 &=& \int_\R |y|  P\left[ \xi_i \le u\big(1-a(\eta_i;\theta)\big)-y~b(\eta_i;\theta); ~i=1,\ldots,m  \right]f_{Y'_{\theta}}(y)dy  \nonumber\\
 &=&  \int_\R |y|  F_{\xi^{(m)}}\left(u\big(1-a(\eta_i;\theta)\big)-y ~b(\eta_i;\theta); ~i=1,\ldots,m  \right)f_{Y'_{\theta}}(y)dy \qquad
 \end{eqnarray}
using the independence of $\xi$ and $Y'_\theta$ in the second equality.\\
Taking the limit as $m\to\infty$ in the previous equations and applying Lemma~\ref{khintchinExt} provide the result (\ref{regresDiscreteEq}). \hfill $\Box$
%%%%%

\begin{example}
Let us consider a stationary and isotropic Gaussian process $X$, with correlation function $r$  defined on $\R^2$ by 
$$r(x)=e^{-||x||^2/2}$$ 
Then , for $x=(x_1,x_2)$, we have:
\begin{eqnarray*} 
&&\partial_{10}r(x)=-x_1r(x); \quad \partial_{01}r(x)=-x_2r(x); \quad \partial_{11}r(x)=-x_2\partial_{10}r(x)=-x_1\partial_{01}r(x); \\
&&\partial_{20}r(x)=(x_1^2-1)r(x); \quad \partial_{02}r(x)=(x_2^2-1)r(x)
\end{eqnarray*}
hence the variance of $Y'_\theta$ becomes
$$ 
\E(Y'^{~2}_\theta)=1, \,\,\forall \theta\in [0,l_1+l_2]
$$ 
and the coefficients $a(.,.)$ and $b(.,.)$ satisfy
\begin{eqnarray*} 
a(\alpha,\theta)&=& a(\theta,\alpha)\\
&=& \left\{
\begin{array}{lcl}
\exp\left\{-\frac12 (\alpha-\theta)^2\right\} &\mbox{if} & 0\le \theta,\alpha\le l_1 ~ \mbox{or if }~  \theta,\alpha \ge l_1\\
&&\\
\exp\left\{-\frac12 (\alpha-\theta)^2-4|(l_1-\alpha)(l_1-\theta)|\sin^2\tilde\varphi \right\} &\mbox{if} & 0\le \theta\le l_1\le \alpha\le l_1+l_2 \\
&\mbox{or if } &  0\le \alpha\le l_1\le \theta\le l_1+l_2
\end{array}
\right. \\
 b(\alpha,\theta)&=& \left\{ 
\begin{array}{l}
(\theta-\alpha)\exp\left\{-\frac12 (\alpha-\theta)^2\right\} \quad \mbox{if} ~  0\le \theta,\alpha\le l_1\quad \mbox{or if }~  \theta,\alpha \ge l_1\\
~\\
%%%
\big(\theta-\alpha+2(\alpha -l_1)\cos^2\tilde\varphi \big)\exp\left\{-\frac12 \left[(\alpha-\theta)^2 +4(\alpha-l_1)(2\alpha-l_1-\theta) \cos^2\tilde\varphi \right]\right\} \\
\qquad\qquad\qquad\qquad \qquad\qquad \qquad\qquad\qquad\qquad   \qquad\qquad  \mbox{if} ~ 0\le \theta\le l_1\le \alpha\le l_1+l_2\\
%%%
-\big(\theta-\alpha+2(\alpha -l_1)\cos^2\tilde\varphi \big)\exp\left\{-\frac12 \left[(\alpha-\theta)^2 +4(\alpha-l_1)(2\alpha-l_1-\theta) \cos^2\tilde\varphi \right]\right\} \\
\qquad\qquad \qquad\qquad \qquad\qquad \qquad\qquad  \qquad\qquad \qquad\qquad   \mbox{if}  \quad 0\le \alpha\le l_1\le \theta\le l_1+l_2 .
\end{array}\right.  
 \end{eqnarray*}
Therefore (\ref{approximEq}) can be computed numerically when replacing
$f_{Y'_{\theta}}$  by a standard normal density function and  $\xi^{(m)}= (\xi_1,\cdots,\xi_m)$ by a  Gaussian ${\cal N}(0, \Sigma_m)$ with the covariance matrix $\Sigma_m$
given by
$$
var(\xi_i)= 1-a^2(\eta_i,\theta)-b^2(\eta_i,\theta) \quad (i=1,\cdots,m)
$$
and, for $1\le i\ne j\le m$, for $\eta_i\ne\eta_j$,
$$
cov(\xi_i,\xi_j)= a(\eta_i,\eta_j)-a(\eta_i,\theta)a(\eta_j,\theta)-b(\eta_i,\theta)b(\eta_j,\theta).
$$
\end{example}

\subsection{Joint distribution for $k$ line segments via a growing circle}

We can extend to $k$ segments what has been previously developed for two ones, considering a  growing circle of radius $t>0$, with center in 0, under the same assumptions on $X$. 
Let be $v_1,\ldots , v_k \in {\cal S}^1$, denoting $k$ directions, and $\varphi_j$ be the angle between $(oe_1)$ and $(ov_j)$: 
\begin{equation}\label{phi_j}
\varphi_j=\angle{(oe_1, ov_j)}, \; j=1,\ldots,k.
\end{equation} 
 Then $X_{tv_j}=X_{t\cos\varphi_j,t\sin\varphi_j}$. \\
For $l_1,\ldots , l_k>0$, we define the union of segments $K=\bigcup_{i=1}^k [0,l_i v_i]$.  
The method consists in introducing a circle and making it grow  with $t$ until meeting a $u$-crossing by $X_s$, for $s\in K$.\\
Setting $\Theta_t =\left \{s=(s_1,\ldots,s_k)\in K: \sum_{i=1}^k s_i^2 \le t^2\right\}$, we can write (analogously to (\ref{expectY}), using Rice formula)
\begin{equation}\label{expectXk}
P[L_1>l_1,\ldots, L_k>l_k]= 1-\sum_{i=1}^k \int_0^{l_i} \E\big[|\partial_{v_i}X_{tv_i}| ~1_{\big(X_s\le u, \forall s\in\Theta_{t}\big)} ~/~X_{tv_i}=u\big] f_{X_{tv_i}}(u)dt . \quad
\end{equation}
Now let us compute the conditional expectation, denoted by $E_i(t)$, appearing as an integrand in (\ref{expectXk}).
We can write, for fixed $i$ and $t\le l_i$,
\begin{eqnarray} \label{condExpect-i}
E_i(t)& = &   \E\big[|\partial_{v_i}X_{tv_i}| ~1_{\big(X_s\le u, \forall s\in\Theta_{t}\big)} ~/~X_{tv_i}=u\big] \nonumber \\
&=& \E\big[|\partial_{v_i}X_{tv_i}| ~1_{\big(X_{hv_i}\le u, ~\forall h\le t\big)}~1_{\big(X_{hv_j}\le u, \forall h\le min(l_j,t),~\forall  j\ne i\big)}~/~X_{tv_i}=u\big]\nonumber\\
&=& \E\big[|\partial_{v_i}X_{tv_i}| ~1_{\big(X_{hv_j}\le u,\, \forall h\le min(l_j,t),~\forall  j=1,\ldots,k \big)}~/~X_{tv_i}=u\big]
\end{eqnarray}
since, for $j=i$, $\min(l_i,t)=t$.\\
Once again, we proceed by standard Gaussian regression, regressing $X_{hv_j}$ on $\big(X_{tv_i}, \partial_{v_i}X_{tv_i}\big)$ at given $h$, $i$ and $t$, for any $j=1,\ldots,k$. So we consider
\begin{eqnarray}
X_{hv_j} &=& Z_{h,j}+\alpha_h^j X_{tv_i}+\beta_h^{j} ~ \partial_{v_i}X_{tv_i} \label{Xvireg}\\
 \mbox{with} \quad 
 \alpha_h^{j} & = & r\big(tv_i-hv_j \big), \nonumber\\
\beta_h^{j} & = &  \cos\varphi_i ~\partial_{10} r\big(tv_i-hv_j \big) + \sin\varphi_i ~\partial_{01}r\big(tv_i-hv_j \big),  \nonumber \\
Z_{h,j} &:& \mbox{independent of}~ (X_{tv_i}, \partial_{v_i}X_{tv_i}), ~\mbox{Gaussian, mean}~0,~ var(Z_{h,j})= 1-(\alpha^j_h)^2-(\beta^j_h)^2 \nonumber\\ 
&&  \mbox{and}~ \E[Z_{h,j} Z_{l,n}] = \E[X_{hv_j}X_{lv_n}] - \alpha_h^{j}\alpha_l^{n} - \beta_h^{j}\beta_l^{n} = r\big(hv_j-lv_n\big)  - \alpha_h^{j}\alpha_l^{n} - \beta_h^{j}\beta_l^{n}. \nonumber
\end{eqnarray}
Notice that we took $Z_{h,j}= Z^{i,t}_{h,j}$, $\alpha_h^{j}=\alpha_h^{i,j}$ and  $\beta_h^{j}=\beta_h^{i,j}$ to simplify the notations when working at given $i$ and $t$.\\
The conditional expectation (\ref{condExpect-i}) can be written as
\begin{eqnarray*}
E_i(t)&=&   \E\big[|\partial_{v_i}X_{tv_i}| 
1_{\big(Z_{h,j}+\alpha_h^j X_{tv_i}+\beta_h^{j} ~ \partial_{v_i}X_{tv_i}\le u, ~\forall h\le min(l_j,t),~\forall  j=1,\ldots,k \big)}~/~X_{tv_i}=u\big] \\
&=&   \E\big[|\partial_{v_i}X_{tv_i}| 
1_{\big(Z_{h,j}+\beta_h^{j} ~ \partial_{v_i}X_{tv_i}\le u(1-\alpha_h^j), ~\forall h\le min(l_j,t),~\forall  j=1,\ldots,k \big)}\big] \\
%%%%%%
&=& \E \left( \E\big[|\partial_{v_i}X_{tv_i}| 1_{\big(Z_{h,j}+\beta_h^{j} ~ \partial_{v_i}X_{tv_i}\le u(1-\alpha_h^j),~ \forall h\le min(l_j,t),~\forall  j=1,\ldots,k \big)}\big] / ( Z_{h,j})_{h\le t, 1\le j\le k}\right)
\end{eqnarray*}
using the independence of ($X_{tv_i}$, $\partial_{v_i}X_{tv_i}$, $(Z_{h,i})$).\\
Now we can evaluate $E_i(t)$ via discretization and using once again the above mentioned independence.
We discretize equidistantly the interval $\displaystyle [0,\max_{1\le i\le k} l_i]$, as $\displaystyle [0,h_1]\cup \big(\cup_{i=1}^{n-1}(h_i,h_{i+1}]\big)$ with $\displaystyle h_n=\max_{1\le i\le k} l_i$ and introduce the corresponding Gaussian vector 
$Z^{(n)}=(Z_{h_m,j}; 1\le m\le n, 1\le j\le k)$ with d.f. $f_{Z^{(n)}}$ and cdf $F_{Z^{(n)}}$. Note that we apply the same discretization in any direction $v_i$, $i=1,\ldots,k$.\\
Then Lemma~\ref{khintchinExt}  can be applied to the $k$ segments, substituting $I(\theta)$ by $[0, l_i v_i]$, and $\eta_1,\ldots,\eta_m$ by $0, h_1 v_i, \ldots, h_{m_i}v_i$, with $h_{m_i}\le l_i< h_{m_i+1}$, for $i=1,\ldots,k$. We obtain
\begin{eqnarray*}
E_i(t)&=&\lim_{n\to\infty} \int_{\R^{n\times k}} \E\big[|\partial_{v_i}X_{tv_i}| ~
1_{\big(\beta_{h_m}^{j} ~ \partial_{v_i}X_{tv_i}\le u(1-\alpha_{h_m}^j)-z_{h_m,j}, \forall h_m\le \min( l_j,t),~\forall  j=1,\ldots,k \big)}\big] f_{Z^{(n)}}(z)dz \\
&=& \lim_{n\to\infty} \int _\R|y| \left(\int_{\R^{n\times k}} 
1_{\big(z_{h_m,j}\le  u(1-\alpha_{h_m}^j)- y\beta_{h_m}^{j} , \forall h_m\le \min( l_j,t),~\forall  j=1,\ldots,k \big)} f_{Z^{(n)}}(z)dz\right) f_{\partial_{v_i}X_{tv_i}} (y)dy \\
&=& \lim_{n\to\infty}  \int_\R |y| F_{Z^{(n)}}\big(w(y,u,\alpha,\beta,t)\big)~f_{\partial_{v_i}X_{tv_i}} (y)dy
\end{eqnarray*}
where $f_{\partial_{v_i}X_{tv_i}}$ denotes the d.f. of $\partial_{v_i}X_{tv_i}$,
\begin{eqnarray}\label{wij}
&& \mbox{and } w(y,u,\alpha,\beta,t) \, \mbox{is a } n\times k-\mbox{matrix having components }\\
&& (w_{mj} ; ~1\le m\le n, 1\le j \le k)\, \mbox{given by} \,
w_{mj} = \left\{\begin{array}{ll}
u(1-\alpha_{h_m}^j)- y\beta_{h_m}^{j} & \mbox{if}~h_m\le \min( l_j,t) \\
+\infty &\mbox{otherwise}.
\end{array}\right. \nonumber
\end{eqnarray}
We can conclude to the following result:
\begin{theorem}
Let $X$ be a stationary Gaussian random field, mean 0 and variance 1, with $C^1$ paths
and a  twice differentiable correlation function $r$.   Then 
$$
P[L_1>l_1,\ldots, L_k>l_k]= 1-\lim_{n\to\infty} \sum_{i=1}^k \int_0^{l_i}\left(\int_\R |y| F_{Z^{(n)}}\big(w(y,u,\alpha,\beta,t)\big)~f_{\partial_{v_i}X_{tv_i}} (y)dy\right) f_{X_{tv_i}}(u)dt
$$
where $w$ is defined in (\ref{wij}).
\end{theorem}
Note that we can deduce from this result a way to evaluate numerically the joint distribution $P[L_1>l_1,\ldots, L_k>l_k]$, as done in Corollary~\ref{numericalApprox}.

\section{The  second moment measure}\label{redmom}

\n Now we describe a method to determine the  second moment measure of the length measure of the boundary $\partial A_u$, as defined in (\ref{redsec}).
It is based on the classical Crofton formula of integral geometry which is widely used in stereology. It allows to determine the length of a planar curve by an integral of the  number of intersection points of the curve with "test" lines, and the integration goes over all lines of the plane w.r.t. a motion invariant measure on the set of lines.
Note that this second moment measure has been studied in \cite{az/wsc} (see Theorem 6.9 and the associated comment p. 181), using another approach, namely the co-area formula. \\

\n Denote by $ G$ the set of all lines in the plane. 
The  $\sigma$-algebra ${\cal G}$ on $G$ is induced by an appropriate parametrization and the Borel $\sigma$-algebra on the parameter space.
Further, $dg$ denotes the element of the measure on $(G,{\cal G})$ which is invariant under translation and rotation of the plane, and normalized such that
$\int ~1 \{ g\cap A \not= \emptyset \}\, dg =2\pi $, for the unit circle $A\subset \R^2$.

\n Let $C(g\cap B)$ denote the number of crossings of $u$ by $X$ on the line $g$ within a set $B\subset\R^2$.

\begin{theorem}\label{pairsectionEIK} 
Let $X$ be a stationary Gaussian random field, mean 0 and variance 1, with $C^1$ paths.
Assume $\partial A_u$ to be smooth (in the sense that it can be parametrized by a $C^1$ mapping).
Then, for bounded Borel sets $B_1,B_2\subset \R^2$, for which $g_1 \cap B_1$ and $g_2 \cap B_2$ consist of finitely many line segments for all lines $g_1, g_2$, 
 we have
\begin{eqnarray}\label{stereolindicatorTh}
\mu ^{(2)}(B_1\times B_2) & = & \frac{1}{4}~ \int \int \E\left[ C(g_1 \cap B_1) \cdot C(g_2 \cap B_2) \right]\, dg_1\, dg_2\,.
\end{eqnarray}
For $g_1\neq g_2$ and not parallel, denote $p\in\R^2$ such that $\{p\}=g_1\cap g_2$, and consider $v_1, v_2\in {\cal S}^1$ with $v_1\neq v_2$ such that $g_1=\R v_1+p,\ g_2=\R v_2+p$. Then 
the expectation appearing as the integrand in (\ref{stereolindicatorTh}), is  given by
\begin{eqnarray}\label{twolinesindicator} 
&&\E \left[ C( g_1 \cap B_1) \cdot C(g_2 \cap B_2) \right] = \\   
&& \!\!\!\!\!  \int \!\!\!\!\! \int  \E[|\partial_{v_1}X_{sv_1}\cdot \partial_{v_2}X_{tv_2} | /
X_{sv_1}=X_{tv_2} =u] ~f_{X_{sv_1},X_{tv_2}}(u,u) \, 1_{B_1-p}(sv_1)  1_{B_2-p}(tv_2) \, ds \,  dt \nonumber
\end{eqnarray} 
where $f_{X_{sv_1},X_{tv_2}}$ denotes the d.f. of $(X_{sv_1}, X_{tv_2})$.
\end{theorem}

\n Comments:
\begin{itemize}
\item The product $\partial_{v_1} X_{sv_1}\cdot \partial_{v_2} X_{tv_2}$ may  again be treated, using  Gaussian regression given in (\ref{Xvireg}), but it will not provide a simpler covariance matrix as the one of ($\partial_{v_1} X_{sv_1}$, $\partial_{v_2} X_{tv_2}$) that we computed using (\ref{partial-r}).
\item Sufficient conditions can be given on $X$ and $u$ for $\partial A_u$ to be smooth. We refer to \cite{az/wsc}, $\S 6.2.2$ or \cite{ad/tay},  $\S 6.2$ .
\end{itemize}

\begin{proof} %of Theorem \ref{pairsectionEIK} 

\n The proof of Theorem \ref{pairsectionEIK} is based on two main steps. The first one is an application of the second-order stereology for planar fibre processes proposed in \cite{wei/na}. 
The second one follows the approach developed in Theorem~\ref{prop2chords}.

\n Applying Theorem 3.1 in \cite{wei/na} for $ \partial A_u$  yields
\begin{eqnarray*}%\label{stereolindicator}
\mu ^{(2)}(B_1\times B_2) &= & \frac{1}{4}~ \E \left(\int \int \sum_{y\in \partial A_u \cap g_1} \, \sum_{z\in \partial A_u \cap g_2} ~1_{B_1\times B_2}(y,z)\, dg_1\, dg_2\right) \\
&= & \frac{1}{4}\int \int \E \left[ C( g_1 \cap B_1) \cdot C(g_2 \cap B_2) \right] \, dg_1\, dg_2 \, .
\end{eqnarray*}
Note that integrating on the restricted domain $\{g_1=g_2\} \cup \{g_1\parallel g_2\} $ would give 0 for the double integral and therefore we consider integration only on $\{g_1\neq g_2\}$ $\cap$ $\{g_1$ not parallel to $g_2\}$.\\ 

\n According to the assumption on $B_i$'s, we can write  $g_i\cap B_i=\bigcup_{j=1}^{n_i} I_{ij}$, for $i=1,2$, and $n_i\in {\mathbb N}$, where the $I_{ij}$ are pairwise disjoint intervals.
Then we obtain
$$
\E \left[ C( g_1 \cap B_1) \cdot C(g_2 \cap B_2) \right] = \sum_{j=1}^{n_1}\sum_{k=1}^{n_2}\E \left[ C(I_{1j}) \cdot C(I_{2k}) \right] \, .
$$

\n Let us compute each term of the double sum. For fixed $j,k$, first shift and rotate $g_1, g_2, B_1, B_2$ such that the lines have a representation $g_i=\R v_i$, $i=1,2$, with $v_1, v_2$ as in (\ref{tildephi}).  
Let $\tilde B_i$ and $\tilde I_{1j}, \tilde I_{2k}$  denote the adequate transformations of $B_i$ and $\ I_{1j}, I_{2k}$, respectively.
Then, using  the diffeomorphism $\rho$ analogous to (\ref{rho}), which may also be applied if the intervals do not intersect, and applying Rice type formula for 2nd moment  (see (6.28) in \cite{az/wsc}), 
provide
\begin{eqnarray*}
&&\E \left[ C(I_{1j}) \cdot C(I_{2k}) \right]\, = \\
&& \int_{\tilde I_{1j}\times \tilde I_{2k} } \!\!\! \E~[|\partial_{v_1} Y_{\theta_1}\cdot \partial_{v_2} Y_{\theta_2} |/  Y_{\theta_1}=Y_{\theta_2}=u] ~f_{Y_{\theta_1},Y_{\theta_2}}(u,u)
1_{\tilde B_1\times \tilde B_2}\left(\rho(\theta_1), \rho(\theta_2)\right)d\theta_1 d\theta_2\, .
\end{eqnarray*}
Note that the rotation has been introduced only to apply (\ref{rho}); what does matter is the shift by $p$, the intersection point of $g_1$ and $g_2$.\\
Combining those results provides the Theorem.
\end{proof}

\section*{Acknowledgments} 
The authors acknowledge the support from Deutsche Forschungsgemeinschaft (DFG, grant number WE 1899/3-1) and from MAP5 UMR 8145 Universit{\'e} Paris Descartes.  This work has been carried out during the first author's stay at the Institut f\"ur Stochastik, Friedrich-Schiller-Universit\"at, Jena, and the second author's stay at MAP5, Universit{\'e} Paris Descartes.  The authors would like to thank both the institutions for hospitality and facilitating the research undertaken. Partial support from RARE-318984 (an FP7 Marie Curie IRSES Fellowship) is kindly acknowledged.

\end{document}